\documentclass[11pt]{article}
\usepackage{amsmath, amsthm, amscd, amsfonts, amssymb, graphicx, color}
\usepackage[margin=2.8cm]{geometry}
\usepackage[bookmarksnumbered, colorlinks, plainpages]{hyperref}

\newtheorem{theorem}{Theorem}[section]
\newtheorem{lemma}[theorem]{Lemma}
\newtheorem{proposition}[theorem]{Proposition}

\theoremstyle{definition}

\newtheorem{example}[theorem]{Example}

\theoremstyle{remark}

\numberwithin{equation}{section}

\DeclareMathOperator{\supp}{supp}
\DeclareMathOperator{\Ann}{Ann}

\allowdisplaybreaks
\begin{document}

\title{\textbf{The ascending chain condition for principal left or right ideals of skew generalized power series rings}}

\author {\bf  F. Padashnik, A. Moussavi and H. Mousavi}
\date{}
\maketitle
\begin{center}

{\small Department of Pure Mathematics, Faculty of Mathematical
Sciences,\\ Tarbiat Modares University, Tehran, Iran, P.O. Box:
14115-134.{\footnote {\noindent Corresponding author.
moussavi.a@modares.ac.ir and  moussavi.a@gmail.com.\\
\indent f.padashnik@modares.ac.ir \\
\indent h.moosavi@modares.ac.ir .}}
}\\
\end{center}
\begin{abstract}
Let $R$ be a ring, $(S,\leq)$ a strictly ordered monoid and $\omega:
S\rightarrow End(R)$ a monoid homomorphism. In this paper we study the ascending chain conditions on principal left (resp. right) ideals of the skew generalized
power series ring $R[[S,\omega ]]$.
 Among other results, it is shown that $R[[S,\omega ]]$   is
 a right archimedean reduced ring if  $S$ is an Artinian
strictly totally ordered monoid,  $R$ is a right archimedean and
$S$-rigid ring which satisfies the ACC on annihilators and
$\omega_s$ preserves nonunits of $R$ for each $s\in S$.
 As a  consequence we deduce that the
power series rings, Laurent series rings, skew power series rings,
skew Laurent series rings and generalized power series rings
are reduced satisfying the ascending chain condition on principal
left (or right) ideals.
 It is also proved that, the skew Laurent polynomial ring $R[x,x^{-1};\alpha]$ satisfies \emph{ACCPL(R)},
  if $R$ is $\alpha$-rigid and satisfies \emph{ACCPL(R)} and the $ACC$ on left(resp. right) annihilators.
   Examples are provided to illustrate and delimit our results.
\end{abstract}

\textit{Key words:} skew generalized power
series ring;  right archimedean ring; annihilator; Artinian strictly
totally ordered monoid; skew Laurent series ring; skew Laurent
polynomial ring.?
?
?\textit{?subjclass: 16D15; 16D40; 16D70}?

\section{Introduction}

Throughout this paper all monoids and rings are with identity
element that is inherited by submonoids and subrings and preserved
under homomorphisms, but neither monoids nor rings are assumed to be
commutative.\par

A commutative ring $R$ is said to satisfy the ascending chain condition for
principal ideals (\emph{ACCP}), if there does not exist an infinite strictly ascending chain of principal ideals of $R$ (see, for example, Dumitrescu et al., \cite{Dumitrescu} or Frohn, \cite{frohnACC}).
The \emph{ACCP} is also called $1$-\emph{ACC} in Frohn \cite{counter}. Clearly every Noetherian ring satisfies \emph{ACCP}. Ribenboim \cite{Ribenboim} gave a sufficient condition for the ring $R[[S]]$ of generalized power series being Noetherian. Varadarajan \cite{Varadarajan} studied Noetherian generalized power series rings. Frohn \cite{counter},
gave an example to show that \emph{ACCP} does not rise to the power series ring in general.
In Dumitrescu et al. (\cite{Dumitrescu}, Proposition 1.2) and Anderson et al. (\cite{andersonI}, Proposition
1.1) the authors gave a necessary and sufficient condition under which the rings
$A+XB[[X]]$ and $A+XB[X]$ satisfy \emph{ACCP} where $A\subseteq B$ are domains and $X$ is an indeterminate. \par

 A partially ordered set $(S, \leq
)$ is called \textit{Artinian} if every strictly decreasing sequence
of elements of $S$ is finite, and $(S, \leq )$ is called
\textit{narrow} if every subset of pairwise order-incomparable
elements of $S$ is finite. Thus, $(S, \leq )$ is Artinian and narrow
if and only if every nonempty subset of $S$ has at least one but
only a finite number of minimal elements.

Clearly, the union of a finite family of artinian and narrow subsets of an ordered set as well as any
subset of an artinian and narrow set are again artinian and narrow.
An \emph{ordered monoid} is a pair $(S, \leq )$ consisting of a
monoid $S$ and an order $\leq $ on $S$ such that for all $a, b, c\in
S$, $a \leq b$ implies $ca\leq cb$ and $ac\leq bc$. An ordered
monoid $(S, \leq )$ is said to be \emph{strictly ordered} if for all
$a, b, c\in  S$, $a < b$ implies $ca < cb$ and $ac < bc$.\par For a
strictly ordered monoid $S$ and a ring $R$,  Ribenboim
\cite{Ribenboim-semisimple} defined the ring of generalized power
series $R[[S]]$ consisting of all maps from $S$ to $R$ whose support
is Artinian and narrow  with  the pointwise addition and the
convolution multiplication. This construction provided interesting
examples of rings (e.g., Elliott and Ribenboim,
\cite{Elliott-Ribenboim 1992}; Ribenboim, \cite{Ribenboim
(1995a)},\cite{Ribenboim (1995b)}) and it was extensively studied by
many authors.\par

 In  \cite{zim}, R. Mazurek and M. Ziembowski, introduced a ``twisted" version of the Ribenboim
construction and study when it produces a von Neumann regular ring.
Now we recall the construction of the skew generalized power series
ring introduced in \cite{zim}. Let $R$ be a ring, $(S, \leq )$ a
strictly ordered monoid, and $\omega: S\rightarrow End(R)$ a monoid
homomorphism. For $s \in S$, let $\omega_{s}$ denote the image of
$s$ under $\omega$, that is $\omega_{s}=\omega (s)$. Let $A$ be the
set of all functions $f: S\rightarrow R$ such that the support
$\supp(f)= \{ s\in  S : f(s)\neq 0 \}$ is Artinian and narrow. Then
for any $s \in S$ and
$f,g \in A $ the set\\
\indent\indent\indent\indent\indent\indent\indent$X_{s}( f, g)=\{(x, y) \in \supp(f) \times \supp(g):  s=xy\}$\\
is finite. Thus one can define the product $f g: S\rightarrow R$ of
$f, g \in A$ as follows: \\
\indent\indent\indent\indent\indent\indent\indent\indent\indent$fg(s)=\sum
\limits_{ (u,v)\in X_{s}(f,g)}f(u)\omega_{u} (g(v)),$\\ (by
convention, a sum over the empty set is $0$). With pointwise
addition and multiplication as defined above, $A$ becomes a ring,
called \textit{the ring of skew generalized power series} with
coefficients in $R$ and exponents in $S$ (one can think of a map
$f:S\rightarrow R$ as a formal series $\sum\limits_{s\in S}r_ss,$
where $r_s = f(s)\in R$) and denoted either by
$R[[S^{\leq},\omega]]$ or by $R[[S,\omega ]]$ (see \cite{unified}
and \cite{von}).
For every $r\in R$ and $s\in S$ we associate the maps $c_r,e_s:S\longrightarrow R$ defined by
\begin{align}\label{e_s}
c_r(x)=\begin{cases}
r\quad; x=1 \\
0 \quad ; \text{Otherwise}
\end{cases},
e_s(x)=\begin{cases}
1\quad; x=s \\
0 \quad  ; \text{Otherwise}
\end{cases}
\end{align}
where $x\in S$.
In fact, $c_r(x)$ and $e_s(x)$ are like $r$ and $x^s$ in $R[x]$ respectively. \\

A ring $R$ is said to satisfy the ascending chain condition on
principal left ideals (\emph{ACCPL}) if there does not exist an
infinite strictly ascending chain of principal left ideals of $R$.
Rings satisfying the ascending chain condition on principal right
ideals (\emph{ACCPR}) are defined analogously.\par

Obvious examples of rings satisfying \emph{ACCPL} are left
Noetherian rings. Also every left perfect ring satisfies
\emph{ACCPL}, since by a celebrated theorem of Bass (see
\cite{Bass}) the left perfect condition is equivalent to the
descending chain condition on principal right ideals, which in turn
implies \emph{ACCPL}, by Jonah's theorem from \cite{Jonah}. In the
commutative case the ascending chain condition on principal ideals
(\emph{ACCP}) appears naturally in studies of factorization in
domains (e.g., \cite{andersonI,Dumitrescu}; see also \cite{Cohn},
Section 2]). For commutative rings several authors studied the
passage of \emph{ACCP} to some classical ring constructions such as
localizations (e.g. \cite{andersonII,Grams,HeinzerAcc}), polynomial
rings (e.g. \cite{frohnACC,frohnACCP,lantz}), monoid rings (e.g.
\cite{Kim}) or power series rings (e.g. \cite{counter}). In
\cite{liu} the \emph{ACCP} condition for commutative generalized
power series rings was studied and it was proved that if $R$ is a
commutative domain and $S$ is a commutative strictly totally ordered
monoid, then the ring $R[[S]]$ of generalized power series with
coefficients in $R$ and exponents in $S$ satisfies \emph{ACCP} if
and only if $R$ and $S$ satisfy \emph{ACCP} (see \cite[Theorem
3.2]{liu}).\\

Frohn in \cite{frohnACC}
showed that  if a ring $R$ satisfies \emph{ACCP} and $R[X]$ has \emph{ACC} on
annihilator ideals, then $R[X]$ also satisfies \emph{ACCP} for commutative
rings. \par

R. Mazurek and M. Zimbowski \cite{zim} proved  that, if $R$ is a
domain and $\omega$ an endomorphism of  $R$, then  $R$ satisfies
\emph{ACCPL} and $\omega_s$  is injective  for each $s\in S$
 if and only if $R[[S,\omega]]$ is a domain and satisfies \emph{ACCPL}; and $R$ is an \emph{ACCPR}-domain and $\omega_s$ is injective for every
$s\in S$ and preserves nonunits of $R$ if and only if
$R[[S,\omega]]$ is an \emph{ACCPR}-domain.\par
Nasr-Isfahani  \cite{nasr} extended Frohn's theorem to the
ring $R[x; \alpha, \delta]$ with some conditions on $R$ and
$\alpha$.\par

According to Krempa \cite{krem}, an endomorphism $\alpha$ of a ring $R$ is
said to be \textit{rigid} if $a\alpha(a)=0$ implies $a=0$, for $a
\in R$. A ring $R$ is said to be $\alpha$-\textit{rigid} if there
exists a rigid endomorphism $\alpha$ of $R$. Clearly, every domain
$D$ with a monomorphism $\alpha$ is rigid. It is clear that, $\alpha$-rigid
rings are \emph{reduced} (rings with no non-zero nilpotent
elements) see e.g., \cite{moussavi}.\par

A ring $R$ is left \textit{archimedean}, if and only if for each
nonunit element $r\in R$ we have $\bigcap_{n\in \mathbb{N}}Rr^n=
\{0\}$. Right archimedean rings are defined similarly.\par

It is well known that each \emph{ACCP}-domain is archimedean. In
fact, if $R$ is a domain, then, by Z. Liu \cite{liu}, for any domain
$R$,  $R$ satisfies \emph{ACCPL} if and only if $\bigcap_{n\in
\mathbb{N}} r_1r_2 \cdots r_nR = \{0\}$ for any sequence
$(r_n)_{n\in \mathbb{N} }$ of nonunits of $R$.\par

In section $2$, we prove that, if $S$ is an Artinian strictly totally ordered monoid,
$R$ is an $S$-rigid ring satisfying \emph{ACCPR}, and $\omega_s$  an
automorphism of $R$ for each $s\in S$, then $R[[S,\omega]]$ is an \emph{ACCPR} ring. Also we show that, if $R$ is an $\alpha$-rigid ring satisfying \emph{ACCPL(R)}  and $ACC$ on left (or right) annihilators, then the skew Laurent polynomial ring $R[x,x^{-1};\alpha]$ satisfies \emph{ACCPL(R)}.\par

In section $3$, we prove that, if $R$ is a right archimedean domain, $S$   an Artinian strictly
totally ordered monoid and $\omega_s$ is injective for any $s\in S$
and preserves nonunit elements of $R$, then $R[[S,\omega]]$ is a
right archimedean domain. We also show that
when $S$ is an Artinian strictly totally ordered monoid,  $R$ is a left archimedean $S$-rigid ring that satisfies \emph{ACC} on annihilators, then $R[[S,\omega]]$ is a left archimedean ring. Also,
when $R$ is a right archimedean $S$-rigid ring that satisfies \emph{ACC} on annihilators and if $\omega_s$ preserves nonunits of $R$ for each $s\in S$, then $R[[S,\omega]]$ is a right archimedean ring.

\section{The Skew Generalized Power Series Rings Satisfying ACCPR or ACCPL}

In this section, we first study the skew generalized power series ring $R[[S,\omega]]$  satisfies \emph{ACCPL} or \emph{ACCPR}, and next we consider
the skew Laurent polynomial ring $R[x,x^{-1};\alpha]$, where $\alpha$ is a monomorphism.
 G. Marks et al.  in \cite[Theorem 4.12]{unified}
proved that if $(S,\le)$ is an a.n.u.p.-monoid, then $R$ is $S$-rigid if and only if $R[[S,\omega]]$ is reduced.\\

The following characterization of \emph{ACCPL}-domains \cite[Lemma
3.1]{liu}, is useful in the sequel.

\begin{proposition}
For any domain $R$, the following are equivalent:

\indent(1) $R$ satisfies \emph{ACCPL}.

\indent(2) For any sequences $(a_n)_{n\in \mathbb{N}}$, $(b_n)_{n\in
\mathbb{N}}$ of elements of $R$ such that $a_n=b_na_{n+1}$ for all
$n\in  \mathbb{N}$, there exists $m\in \mathbb{N}$ with $b_n\in
U(R)$ for all $n\ge m$.

\indent(3) For any sequences $(a_n)_{n\in \mathbb{N}}$, $(b_n)_{n\in
\mathbb{N}}$ of elements of $R$ such that $a_n=b_na_{n+1}$ for all
$n\in \mathbb{N}$, there exists $m\in \mathbb{N}$ with $b_m\in
U(R)$.

\indent(4) $\bigcap_{n\in \mathbb{N}}r_1r_2\cdots r_nR=\{0\}$ for
any sequence $(r_n)_{n\in \mathbb{N}}$ of nonunits of $R$.
\end{proposition}

The following lemmas are very useful in our proofs in this section. Their proof are similar to what Hong et al. proved in \cite[Lemma 4]{ore}.
\begin{lemma}\label{jav}
Let $R$ be a $S$-rigid ring and $a,b \in R$. Then we have the
following

(1) If $ab=0$, then for any $s\in S$ we have $a\omega_{s^n}(b)=\omega_{s^n}(a)b=0$  for any positive integer $n$.

(2) If for some $s\in S$ and positive integer $k$, $a\omega_{s^k}(b)=\omega_{s^k}(a)b=0$, then $ab=0$.
\end{lemma}

\begin{lemma}\label{kav}
Let $R$ be an $S$-rigid ring. If $\omega_s(ab)=0$ with $s\in S$, then
$\omega_{ss'}(a)\omega_s(b)=0$ for any $s'\in S$ such that
$ss'=s's$.
\end{lemma}
\begin{proof}
One can see that
\begin{align*}
\omega_s(b)\omega_{ss'}(a)\omega_{s'}(\omega_s(b)\omega_{ss'}(a))=&\omega_s(b)\omega_{ss'}(a)\omega_{s'}(\omega_s(b))\omega_{s'}(\omega_{ss'}(a))\nonumber\\
=&\omega_s(b)\omega_{ss'}(a)\omega_{s's}(b)\omega_{s'ss'}(a).
\end{align*}
Since $ss'=s's$, we have
\begin{align*}
\omega_s(b)\omega_{ss'}(a)\omega_{s's}(b)\omega_{s'ss'}(a) =\omega_s(b)\omega_{s's}(ab)\omega_{s'ss'}(a)=\omega_s(b)\omega_{s'}(\omega_s(ab))\omega_{s'ss'}(a)=0.
\end{align*}
Since $R$ is $S$-rigid, $\omega_s(b)\omega_{ss'}(a)=0$ so
$(\omega_{ss'}(a)\omega_s(b))^2=0$, and
$\omega_{ss'}(a)\omega_s(b)=0$, as $R$ is reduced.
\end{proof}

 Recall that an ideal $I$ of $R$ is called
$\alpha$-ideal, if $\alpha (I) \subseteq I$. The ideal $I$ is called
$\alpha$-invariant if $\alpha^{-1}(I) = I$. R. Mazurek and M. Zimbowski in \cite{zim} proved  that, if $R$ is a domain  and $\omega_s$  is an injective endomorphism of  $R$ for each $s\in S$, then  $R$ satisfies \emph{ACCPL}  if and only if $R[[S,\omega]]$ is a domain and satisfies \emph{ACCPL}.\par
  In the following result we consider the case that $R$ is not assumed to be a domain. There are various monoids $S$ and $S$-rigid rings which are not domains and satisfy \emph{ACCPR}.

\begin{theorem}\label{reza}
Let $S$ be an Artinian strictly totally ordered monoid and $\omega_s$  be an
automorphism of $R$, for each $s\in S$. If
$R$ is an $S$-rigid ring satisfying ACCPR,  then $R[[S,\omega]]$ satisfies ACCPR.
\end{theorem}
\begin{proof}
We use the method employed by Frohn \cite[Theorem 4.1]{frohnACC}. For each $f \in
A=R[[S,\omega]]$ let $I_f=\{g(\pi(g))|g\in AfA\}\cup\{0\}$.
 It is easy to see that $I_f$ is an ideal of $R$. Next assume on the contrary that there exists a non stabilizing chain of principal right ideals of $A$. So the set
\begin{align*}
 M=\left\{\Ann_R(\bigcup_{i\ge1}I_{g_i}) | g_1A\subseteq g_2A \subseteq \cdots  \text{is a non stabilizing chain of principal right
ideals in } A\right\}
\end{align*}
is nonempty. Since R is $S$-rigid, $R$ is reduced. So it is easy to
see that $R$ satisfies the \emph{ACC} on left annihilators. Thus $M$ has a
maximal element. Let $P=\Ann_R(\bigcup_{i\ge 1}I_{f_i})$ be a
maximal element of $M$, where $f_1A\subseteq f_2A\subseteq \cdots$
is a nonstabilizing chain in $A$.

We show that $P$ is a completely
prime ideal of $R$. Assume that $a \notin P , b\notin P$ and $ab\in P$. Since $R$ is $S$-rigid,  using Lemma
\ref{kav} we can see that $a\in \Ann_R(\bigcup_{i\ge1}I_{bf_i})$. Also we have $P\subseteq \Ann_R(\bigcup_{i\ge1}I_{bf_i})$. So
the chain $bf_1A \subseteq bf_2A \subseteq \cdots$ stabilizes. Then there exists a positive integer $t$ such that
for each $n\ge t$, $bf_{n+1}=bf_nh_n$ for some $h_n\in A$. For each positive integer $n$, there
exists $v_n\in A$ such that $f_n=f_{n+1}v_n$. Thus for each $n\ge t$, $bf_{n+1}(1- v_nh_n)=0$.
Let $q_i=f_i(1-v_{i-1}h_{i-1})$, for each $i>t$. Since $R$ is reduced, $b\in \Ann_R(\bigcup_{i\ge1}Iq_i)$ and $P\subseteq_R(\bigcup_{i\ge1}Iq_i)$. Then the chain $q_1A\subseteq q_2A\subseteq \cdots$ stabilizes. Thus there exists
a positive integer $t'$ such that for each $m\ge t'$, $q_{m+1}=q_ml_m$ for some $l_m\in A$.
Then $f_{m+1}(1-g_mh_m)=f_m(1-g_{m-1}h_{m-1})l_m$ and so $f_{m+1}=f_ml_m+f_m(1-g_{m-1}h_{m-1})l_m$. Thus we have a contradiction $f_{m+1}\in f_mA$. Which shows that $P$ is a completely prime ideal of $R$.

Since $R$ is $S$-rigid and $P=\Ann_R(\bigcup_{i\ge1}I_{f_i})$, using
Lemma \ref{kav}, we claim that $P$ is $\omega_s$-invariant for any
$s \in S$. Let $r\in \omega_s^{-1}(P)$. So
$\omega_s^{-1}(p)=r$ for some $p\in P$, and so
$\omega_s(r)(\bigcup_{i\ge1}(I_{f_i}))=0$. Thus $\omega_s(r)u=0$ for
every $u\in \bigcup_{i\ge1}I_{f_i}$. Hence
$\omega_s(ru)=\omega_s(r)\omega_s(u)=0$, and since $\omega_s$ is
injective, $ru=0$. Thus  $r\in \Ann_R(\bigcup_{i\ge1}I_{f_i})=P$
which means that $\omega_s^{-1}(P)\subseteq P$.

Now let $T=(R/P)[[S,\bar{\omega}]]$. Since $R$ is an \emph{ACCPR}-ring and $P$ is a
completely prime ideal of $R$, it follows that $R/P$ is an
\emph{ACCPR}-domain. On the other hand $\omega_s$ is an automorphism, thus
it preserves nonunit element of $R$ for every $s\in S$. Hence $T$ is
an \emph{ACCPR}-domain.

For each positive integer $i$, $\bar{f_i}=\bar{f_{i+1}}\bar{g_i}$,
where $\bar{f}(s)=f(s)+P\in T$. If $\bar{f_i}=\bar{0}$ for some $i$,
then $f_i(\pi(f_i))\in P$. We claim that $f_i(\pi(f_i))^2=0$. One
can see that $f_i\pi(f_i)\in P$ which means that
$f_i(\pi(f_i))(\bigcup_{i\ge1}I_{f_i})=0$. Thus
$f_i(\pi(f_i))f(\pi(1f_i1))=0$ which gives $f_i(\pi(f_i))^2=0$.
Since $R$ is $S$-rigid, $R$ is reduced and $f_i(\pi(f_i))=0$ which
is a contradiction because $\pi(f_i)\in \supp(f_i)$. So for each
$i$, $\bar{f_i}\neq \bar{0}$ and hence $\bar{g_i}\neq \bar{0}$. By
 \cite[Proposition 2.1]{zim}, there exists a positive integer $j$
such that for each $m\ge j$, $\bar{g_m}$ is invertible in $T$. Then
there is some $\bar{h}\in T$ such that
$\bar{g_m}\bar{h}=\bar{h}\bar{g_m}=\bar{1}$.  Hence
$\bar{g_m}\bar{h}-\bar{1}=\bar{0}$. So it is easy to see that for
each $s'\in \supp(b)$, we have $b(s')=(g_mh-1)(s')\in P$. Since
$b(s')\in P$, for each $i$, $b(s')f_i(\pi(f_i))=0$ which means that
$b(s')f_{m+1}(\pi(f_{m+1}))=0$ and hence
$f_{m+1}\big(\pi(f_{m+1})\big)b(s')=0$. By Lemma  \ref{kav}, we have
$f_{m+1}(\pi(f_{m+1}))\omega_{\pi(f_{m+1})}(b(s'))=0$. Now, define
\begin{align*}
B=\{ s | s > \pi(f_{m+1}) , f_{m+1}(s)\omega_s(b(s'))\neq0 \}.
\end{align*}
If $B=\emptyset$, then $f_{m+1}b(s')=0$ and so $f_{m+1}(g_mh-1)=0$. If $B\neq \emptyset$, $B$ has a minimum element $s''$, since $S$ is Artinian. This means that $f_{m+1}(s'')\omega_{s''}(b(s'))\neq0$
 and hence $f_{m+1}(s'')b(s')\neq 0$. On the other hand, we know that $b\in P$ and $f_{m+1}b \in Af_{m+1}A$. We have also $b(s')\in P$, so
\begin{align*}
b(s')\bigg(\big(f_{m+1}b(s')\big)\big(\pi(f_{m+1}b(s'))\big)\bigg)=b(s')\bigg(f_{m+1}(s'')\omega_{s''}\big(b(s')\big)\bigg)=0.
\end{align*}
Since $A$ is reduced, $f_{m+1}(s'')\omega_{s''}(b(s'))b(s')=0$. By
Lemma \ref{kav},
$f_{m+1}(s'')\omega_{s''}(b(s'))\omega_{s''}(b(s'))=0$. Hence
$f_{m+1}(s'')\omega_{s''}(b(s')^2)=0$. Thus by Lemma \ref{kav},
$f_{m+1}(s'')b^2(s')=0$. But $A$ is reduced, so we get
$f_{m+1}(s'')b(s')=0$ which contradicts to the definition of $s''$.
Hence $B=\emptyset$, which yields $f_{m+1}(g_mh-1)=0$. So
$f_{m+1}=f_{m+1}g_mh=f_mh$ and thus the chain $f_1A \subseteq f_2A
\subseteq f_3A \subseteq \cdots $ will stabilize, which is a
contradiction. Thus the result follows.
\end{proof}

 The following example shows that, in Theorems \ref {reza},  \ref{domar} and \ref{rigidar}, the
Artinian condition on the monoid $S$ is not superfluous.

\begin{example}
 Let $\mathbb{Z}_n$ be the ring of integers module $n$. Then $\mathbb{Z}_n$ satisfies \emph{ACCP} and archimedean condition and also  it satisfies the \emph{ACC} on right annihilators. Now consider $A=\mathbb{Z}_n[[\mathbb{Q},id_{\mathbb{Z}}]]$.
We have an ascending chain as follows:
\begin{align}
\langle e_1 \rangle \subsetneq \langle e_{\frac{1}{2}} \rangle \subsetneq \langle e_{\frac{1}{4}} \rangle \subsetneq \langle e_{\frac{1}{8}} \rangle \subsetneq  \cdots
\end{align}which will not be stabilized,  where $e_s$ is defined in \ref{e_s}. So $\mathbb{Z}_n[[\mathbb{Q},id_{\mathbb{Z}}]]$ does not satisfy \emph{ACCP}. Also one can see easily that $e_1\in Ae_{(\frac{1}{2})^n}$ for each $n\in \mathbb{N}$. Hence
\begin{align*}
\bigcap_{n\in \mathbb{N}}Ae_{(\frac{1}{2})^n}\neq \{0\}.
\end{align*}
So $A$ is not archimedean. Note that this example does not
contradicts  with \cite[Theorems 3.1, 3.13]{zim}, as
$(S,+)=\mathbb{Q}^+\cup \{0\}$ does not satisfy \emph{ACCP}.

Now, consider the sequences $s_n=\frac{1}{2^n}$ and $r_n=\frac{1}{2^{n+1}}$. We have the recursive formula $s_n=r_n + s_{n+1}$. So $r_n$ has to be unit. But the units of $S$ is $U(S)=\{0\}$.
So $r_n \notin U(S)$  for any $n\in \mathbb{N}$.
Thus according to \cite[Proposition 2.1]{zim}, $S$ does not satisfy \emph{ACCP}. So $S$ doesn't have the necessary condition of \cite[Theorems 3.1, 3.13]{zim}. Indeed, we can not apply these Theorems for $S=\mathbb{Q}^+ \cup \{0\}$.
\end{example}

The following example,
(See \cite[Example] {lantz}), shows that the $S$-rigid condition in \ref{reza} is not superfluous.

\begin{example}
 Let $k$ be a field and $A_1, A_2, \cdots$ be indeterminates over $k$, and set
\begin{align}
S = \frac{k[A_1, A_2, \cdots]}{(\{A_n(A_n -A_{n-1}) : n \geq 2\})k[A_1, A_2, \cdots]}.
\end{align}
Denote by $a_n$ the image of $A_n$ in $S$ and by $R$ the localization of $S$ at the ideal $(a_1, a_2, \cdots )S$.
Note that $S$ is a limit of the rings $S_n$ where $S_1$ is $S_1 := k[a_1]$ and
\begin{align*}
S_n := S_{n-1}[a_n] =
\frac{S_{n-1}[A_n]}{A_n(a_{n-1} -A_n)S_{n-1}[A_n]}\quad for \quad n \geq 2.
\end{align*}
Heinzer and Lantz in \cite{lantz} proved that $R$ satisfies \emph{ACCP} but the ring $R[x]$, does not satisfy \emph{ACCP}. Note that in $S$
we have
\begin{align*}
a_3^2(a_1 - a_2)^2 = a_3a_2(a_1 - a_2)^2 = 0,
\end{align*}
but $a_3(a_1 - a_2) \neq 0$. Thus $S$ is
not reduced and since $R$ contains (an isomorphic copy of) $S$ (see \cite{lantz}), $R$ is not
reduced. So the $S$-rigid condition in \ref{reza} is not superfluous.
\end{example}

We now consider a case for which $S$ is not positive. Let $R$ be a ring with a monomorphism $\alpha$.
  We denote $R[x;\alpha]$ the
Ore extension whose elements are the polynomials
$\Sigma_{i=0}^{n}r_{i}x^{i}$, $r_{i}\in R$, where the addition is
defined as usual and the multiplication subject to the relation
$xa=\alpha (a)x $ for any $a\in R$. The set $\{x^j\}_{j\geq0}$ is
easily seen to be a left Ore subset of $R[x\ ;\alpha]$, so that one
can localize $R[x ;\alpha]$ and form the skew Laurent polynomial
ring $R[x,x^{-1}\ ;\alpha]$. Elements of $R[x,x^{-1}\ ;\alpha]$ are
finite sums of elements of the form $x^{-j}rx^i$ where $r\in R$ and
$i$ and $j$ are nonnegative integers.\\

 Now we consider D.A. Jordan's construction of the ring
$A(R,\alpha)$ (See [13], for more details). Let $A(R,\alpha)$ be the
subset $\{x^{-i}rx^i\ |\ r\in R\ ,\ i\geq0\}$ of the skew Laurent
polynomial ring $R[x,x^{-1};\alpha]$. For each $j\geq0$,
$x^{-i}rx^i=x^{-(i+j)}\alpha^j(r)x^{(i+j)}$. It follows that the set
of all such elements forms a subring of $R[x,x^{-1};\alpha]$ with
$x^{-i}rx^i+x^{-j}rx^j=x^{-(i+j)}(\alpha^j(r)+\alpha^i(s))x^{(i+j)}$ and\\
$(x^{-i}rx^i)(x^{-j}sx^j)=x^{-(i+j)}\alpha^j(r)\alpha^i(s)x^{(i+j)}$
for $r,s\in R$ and $i,j\geq0$. Note that $\alpha$ is actually an
automorphism of $A(R,\alpha)$. We have $R[x,x^{-1};\alpha]\simeq
A(R,\alpha)[x,x^{-1};\alpha]$, by way of an isomorphism which maps
$x^{-i}rx^{j}$ to $\alpha^{-i}(r)x^{j-i}$.\\

 Now we examine the
\emph{ACCPL} condition for the skew Laurent polynomial ring
$R[x,x^{-1};\alpha]$. First, we recall the following propositions
which are proved in \cite{Jordan,Weakly}.
\begin{proposition}
If $R$ is a domain and $\alpha$ is a monomorphism of $R$, then $A(R,\alpha)$ is a domain.
\end{proposition}

\begin{proposition}\label{pro 2}
If $\alpha$ is monomorphism and $R$ is $\alpha$-rigid, then  $A(R,\alpha)$ is $\alpha$-rigid.
\end{proposition}

\begin{proposition}\label{pro 3}
If $R$ is an ACCPL(R)-ring, then $A(R,\alpha)$ is an ACCPL(R)-ring.
\end{proposition}

\begin{proof}
Assume that
\begin{align*}
\langle x^{-i_1}r_1x^{i_1} \rangle \subseteq \langle x^{-i_2}r_2x^{i_2} \rangle \subseteq \cdots
\end{align*}
is a nonstabilized chain in the ring $A(R,\alpha)$. Then, $x^{-i_l}r_lx^{i_l}=(x^{-i_{l+1}}r_{l+1}x^{i_{l+1}})(x^{-t}sx^t)$ for some
$x^{-t}sx^{t}\in A(R,\alpha)$.

Hence $\alpha^{i_{l+1}+t}(r_l)=\alpha^{i_l+t}(r_{l+1})\alpha^{i_l+i_{l+1}}(s)$. If $s'=\alpha^{i_l+i_{l+1}}(s)$, then  $r_l=\alpha^{i_l-i_{l+1}}(r_{l+1})s'$. So
\begin{align*}
\langle r_1 \rangle \subseteq \langle \alpha^{i_1-i_2}(r_2) \rangle \subseteq \langle \alpha^{i_1-i_3}(r_3) \rangle\subseteq \cdots.
\end{align*}
Since $R$ is \emph{ACCPL(R)}-ring, the above chain will stabilized. So $r_k=\alpha^{i_k-i_{k+1}}(r_{k+1})u$, where $u$ is a unit. Thus $\alpha^{i_{k+1}}(r_k)=\alpha^{i_k}(r_{k+1})\alpha^{i_{k+1}}(u)$. Consider $u':=\alpha^{i_{k+1}}(u)$. It is easy to see that $u'$ is a unit in $R$, so
\begin{align*}
x^{-i_{k+1}}(r_{k+1})x^{i_{k+1}}=\alpha^{i_{k+1}}(r_{k+1})=\alpha^{i_k}(r_k)u'=(x^{-i_k}(r_k)x^{i_k})(x^{-i_{k+1}}ux^{i_{k+1}}).
\end{align*}
So $\langle x^{-i_{k+1}}r_{k+1}x^{i_{k+1}} \rangle=\langle x^{-i_k}r_kx^{i_k} \rangle$, which implies that the chain $\langle x^{-i_1}r_1x^{i_1} \rangle \subseteq \langle x^{-i_2}r_2x^{i_2} \rangle\subseteq \cdots$ will stabilize.
\end{proof}

\begin{lemma}\label{pro 4}
Let $R$ be an $\alpha$-rigid ring. If $\Ann(L_1)\subseteq \Ann(L_2)\subseteq \cdots$, where $L_i\subseteq A(R;\alpha)$ and  $K_i=\{k| \exists t :x^{-t}kx^t \in L_i\}$, then
\begin{align*}
\Ann(K_1)\subseteq \Ann(K_2)\subseteq \cdots.
\end{align*}
\end{lemma}

\begin{proof}
Let $s\in \Ann(K_l)$. So $sk_l=0$ for each $k_l\in K_l$, so there
exists $c> 0$, such that $x^{-c}k_lx^c\in L_l$. Since $sk_l=0$,
$s\alpha^r(k_l)=0$ for each $r\geq 0$, as $R$ is $\alpha$-rigid. So
$(x^{-r}sx^r)(x^{-c}k_lx^c)=0$. Since $k_l$ is arbitrary in $K_l$,
$x^{-r}sx^r\in \Ann(L_l)$. So $x^{-r}sx^r\in \Ann(L_{l+1})$, for
each $r$. Suppose that $a\in K_{l+1}$. So there is $b>0$ such that
$x^{-b}ax^b\in L_{l+1}$. Then $(x^{-r}sx^r)(x^{-b}ax^b)=0$ for each
$r> 0$. So $x^{-r-b}\alpha^b(s)\alpha^r(a)x^{b+r}=0$, which means
that $sa=0$. So $s\in \Ann(K_{l+1})$, and hence $\Ann(K_l)\subseteq
\Ann(K_{l+1})$.
\end{proof}

\begin{theorem}\label{ACCPL-Domain}
If $R$ is an ACCPR(L)-domain with a monomorphism $\alpha$, then $R[x,x^{-1};\alpha]$ is an ACCPR(L)-domain.
\end{theorem}

\begin{proof}
Let $R$ be an \emph{ACCPR(L)}-domain. Then $A:=A(R,\alpha)$ is an \emph{ACCPR(L)}-domain by Proposition \ref{pro 2}.
By \cite[Section 2]{Jordan},  $R[x,x^{-1};\alpha]\simeq A[x,x^{-1};\alpha]$,
so it is enough to show that $A[x,x^{-1};\alpha]$ is an \emph{ACCPR(L)}-domain.
Let $f\in A[x,x^{-1};\alpha]$. Suppose that $d_+(f)$ and $d_-(f)$ denote the degree of $f$ in positive and negetive coefficients, respectively. Suppose that $\langle f_1 \rangle\subseteq \langle  f_2 \rangle\subseteq \cdots$ is a non-stabilized chain. So $f_i=f_{i+1}k$. Let $\theta_i=d_-(f_i)$ and $\delta_i=d_+(f_i)$. We claim that there are two cases:

$(i)$: $\theta_i\le \theta_{i+1}$.

$(ii)$: $\delta_i\ge \delta_{i+1}$.

To prove it, suppose that $\delta_i<\delta_{i+1}$ and $\theta_i>\theta_{i+1}$. Let $f_i(x)=\sum_{l=\theta_i}^{\delta_i}(x^{-t_l}f_l^ix^{t_l})x^l$ and $k(x)=\sum_{j=d_-(k)}^{d_+(k)}(x^{-t_j}k_jx^{t_j})x^j$ and $f_{i+1}(x)=\sum_{h=\theta_{i+1}}^{\delta_{i+1}}(x^{-t_h}f_h^{i+1}x^{t_h})x^h$. Hence

\begin{align*}
x^{-t_{\delta_i}}f_{\delta_i}^ix^{t_{\delta_i}}=(x^{-t_{\delta_{i+1}}}f_{\delta_{i+1}}^{i+1}x^{t_{\delta_{i+1}}})(x^{-t_{d_+(k)}}k_{d_+(k)}x^{t_{d_+(k)}}).
\end{align*}
Or we have,
\begin{align*}
x^{-t_{\theta_i}}f_{\theta_i}^ix^{t_{\theta_i}}=(x^{-t_{\theta_{i+1}}}f_{\theta_{i+1}}^{i+1}x^{t_{\theta_{i+1}}})(x^{-t_{d_-(k)}}k_{d_-(k)}x^{t_{d_-(k)}}).
\end{align*}
So if $d_+(k)\ge 0$, then $\delta_i\ge \delta_{i+1}$, since $R$ is domain and $f_{\delta_{i+1}}^{i+1}$, $k_{d_+(k)}\neq 0$. Similarly, if $d_-(k)<0$, then $\theta_i\le \theta_{i+1}$. So it contradicts to our assumption, hence one of mentioned cases occurs. Now there are three cases for the sequence $\langle f_1 \rangle\subseteq \langle f_2 \rangle \subseteq \cdots$.

$(i)$: There exists a subsequence $f_{i_j}$ such that $\langle f_{i_1} \rangle \subseteq \langle f_{i_2} \rangle \subseteq \cdots$ and $\delta_{i_j}>\delta_{i_{j-1}}$ for all $j$. But it is impossible unless $\delta_1=+\infty$ which is a contradiction.

$(ii)$: For each $i$ we have $\delta_i=\delta_{i+1}$ and $\theta_i=\theta_{i+1}$. So if $f_{i+1}k_i=f_i$, then $k_i$ is constant. But $A$ is an \emph{ACCPL(R)}-ring, so there exists $N$ such that $k_N$ is unit. Thus $\langle f_{N+1}\rangle=\langle f_N \rangle$.

$(iii)$: There does not exist any subsequence such that $\delta_{i_j}>\delta_{i_{j+1}}$ for all $j$. So it should exist one subsequence $\{f_{i_j}\}$ such that $\theta_{i_j}<\theta{i_{j-1}}$, otherwise $\delta_i=\delta_{i+1}$ and $\theta_i=\theta_{i+1}$ for $i>N$ which is case $(ii)$. This means that $\theta_1=-\infty$ which is impossible. So the result follows
\end{proof}
In the following result we consider the case that $R$ is not assumed to be a domain. Notice that rings with  rigid endomorphisms are reduced.
\begin{theorem}\label{thm 6.7}
Let $R$ be an $\alpha$-rigid ring and $\alpha$ preserves nonunits. If $R$ is an ACCPL(R)-ring satisfying $ACC$ on left(or right) annihilators,  then  $R[x,x^{-1};\alpha]$ is an ACCPL(R)-ring.
\end{theorem}

\begin{proof} Let $R$ be an \emph{ACCPL(R)}-ring satisfying  \emph{ACC} on left(or right) annihilators. So $A:=A(R,\alpha)$ is an \emph{ACCPL(R)}-ring and satisfies the \emph{ACC} for left(or right) annihilators by Propositions \ref{pro 3} and \ref{pro 4}. Because $R[x,x^{-1};\alpha] \simeq A[x,x^{-1};\alpha]$, so  it is enough to show that  $A[x,x^{-1};\alpha]$ is an \emph{ACCPL(R)}-ring.
We prove this  for right case, the left case is similar. Let $I_f:=\{a_{d_+(g)}+b_{d_-(h)}| g,h\in SfS, g=\sum a_ix^i, h=\sum b_ix^i \}$ for each $f\in A[x,x^{-1};\alpha]$. One can show that $I_f$ is an ideal.

Now assume that there exists a nonstabilizing chain of principal  right ideals of $S$. So the set
\begin{align*}
M=\left\{\Ann(\bigcup I_{g_i}) | g_1S\subseteq g_2S\subseteq \cdots  \textit{a nonstabilizing chain} \right\}
\end{align*}
is nonempty. Since $R$ is $\alpha$-rigid, $R$ is reduced and so it is easy to see that since $A$ satisfies the \emph{ACC} on right annihilators, $A$ satisfies the \emph{ACC} on left annihilators. Thus $M$ has a maximal element. Let $P=\Ann(\bigcup I_{f_i})$ be a maximal element of $M$, where $f_1S\subseteq f_2S\subseteq \cdots$ is a nonstabilizing chain in $S$. We show that $P$ is a completely prime ideal in $A$. Assume $a,b\in R\setminus P$ and $ab\in P$. So $ab\in \Ann(\bigcup I_{f_i})$ which  means that $abI_{f_i}=0$ for each $i$. So for each $x^{-i}(a_{d_+(g)}+b_{d_-(h)})x^i$, $abx^{-i}(a_{d_+(g)}+b_{d_-(h)})x^i=0$. If $a=x^{-\gamma}a' x^{\gamma}$, $b=x^{-\beta}b' x^{\beta}$, then
\begin{align*}
x^{-\gamma}a' x^{\gamma}(x^{-\beta}b' x^{\beta}.x^{-i}(a_{d_+(g)}+b_{d_-(h)})x^i)=0
\end{align*}
with $b'g, b'h\in SfS$. Thus,
\begin{align*}
&x^{-\gamma}a' x^{\gamma}(x^{-\beta}b' x^{\beta}.x^{-i}(a_{d_+(g)}+b_{d_-(h)})x^i)=\\
&x^{-\gamma}a' x^{\gamma}(x^{-\beta -i}\alpha^i(b')\alpha^{\beta}(a_{d_+(g)}+b_{d_-(h)})x^{\beta +i})=0.
\end{align*}
Since $A$ is $\alpha$-rigid,  $a\in \Ann(I_{bf_i})$. Also the chain $bf_1S \subseteq bf_2S\subseteq \cdots$ stabilizes, because $P\subseteq \Ann(I_{bf_i})$. Hence there exists a positive integer $t$ such that for each $n\ge t$, $bf_{n+1}=bf_nh_n$. Also for each $n$ one can see that $f_n=f_{n+1}g_n$. So $bf_{n+1}=bf_{n+1}g_nh_n$. So $bf_{n+1}(1-g_nh_n)=0$. Let $q_i=f_i(1-g_{i-1}h_{i-1})$, for each $i>t$. So $b\in \Ann(\bigcup I_{q_i})$ and also $P\subseteq \Ann(I_{q_i})$. So $q_1S\subseteq q_2S\subseteq \cdots$ stabilizes. Thus there exists a positive integer $t'$ such that for each $m \ge t'$, $q_{m+1}=q_ml_m$ for some $l_m\in S$. Then $f_{m+1}(1-g_mh_m)=f_m(1-g_{m-1}h_{m-1})l_m$ and so $f_{m+1}\in f_mS$, which is a contradiction.  So $P$ is completely prime.

We show that $P$ is $\alpha$-invariant. Let $a\in \alpha^{-1}(P)$. So $\alpha(a)\in P$. Hence $\alpha(a)I_{f_i}=0$ for each $i$. But, since $A$ is $\alpha$-rigid,  $aI_{f_i}=0$ for each $i$. So $a\in P$ and that $\alpha^{-1}(P)\subseteq P$.

Let $a\in P$. So $aI_{f_i}=0$ for each $i$. Hence $\alpha^{-1}(a)I_{f_i}=0$, as $R$ is $\alpha$-rigid and $\alpha$ is an automorphism of $A$. So $\alpha^{-1}(P)=P$ and $P$ is $\alpha$-invariant. Thus $T:=R/P[x,x^{-1};\overline{\alpha}]$ is an \emph{ACCPL(R)}-domain. We know that for each $i$, $\bar{f_i}=\overline{f_{i+1}}\bar{g_i}$ with
\begin{align*}
\bar{f_i}=\sum_{t=\theta_i}^{\delta_i}(a_t+P)x^{t}.
\end{align*}
If $\bar{f_i}=\bar{0}$ for some $i$, then $a_{\delta_i}\in P$. So $a_{\delta_i}I_{f_i}=0$. Thus $a_{\delta_i}\alpha^{\delta_i}(a_{\delta_i})=0$ which means that $a_{\delta_i}=0$, since $A$ is $\alpha$-rigid. So $\bar{f_i}\neq \bar{0}$ and so $\bar{g_i}\neq \bar{0}$. Since $T$ is an \emph{ACCPR(L)}-ring, $\bar{g_i}$ must be a unit where $i>m$ for some $m>0$. So $\bar{g_i}\bar{h}=\bar{h}\bar{g_i}=\bar{1}$. So for each coefficient $b$ of the polynomial $g_mh-1$, we claim that $b\in P$. We claim that $f_{m+1}(g_mh-1)=0$. Assume that
\begin{align*}
f_{m+1}=\sum_{t=\theta_{m+1}}^{\delta_{m+1}}(a_t)x^{t}.
\end{align*}
So $a_{\delta_{m+1}},a_{\theta_{m+1}}\in I_{f_{m+1}}$. Since $b\in P$,  $ba_{\delta_{m+1}}=0,ba_{\theta_{m+1}}=0$. Hence
\begin{align}
f_{m+1}b&=\sum_{t=\theta_{m+1}}^{\delta_{m+1}}(a_t)x^{t}b\nonumber\\
&=a_{\theta_{m+1}}\alpha^{\theta_{m+1}}(b)x^{\theta_{m+1}}+\cdots +a_{\delta_{m+1}}\alpha^{\delta_{m+1}}(b)x^{\delta_{m+1}}\nonumber\\
&=\sum_{t=\theta_{m+1}+1}^{\delta_{m+1}-1}(a_t)x^{t}b.
\end{align}
But $f_{m+1}b\in SfS$, so $a_{\theta_{m+1}+1}, a_{\delta_{m+1}-1}\in I_{f_{m+1}}$. So $ba_{\theta_{m+1}+1}=ba_{\delta_{m+1}-1}=0$. Hence inductively we have $f_{m+1}b=0$. So $f_{m+1}(g_mh_m-1)=0$, and $f_{m+1}=f_{m+1}g_mh=f_mh$. So $\langle f_1 \rangle \subseteq \langle f_2 \rangle \subseteq \cdots$ stabilizes, which  contradicts our assumption and the result follows.
\end{proof}

\begin{example}
Let $k$ be a field and $R=\frac{k\langle x_1,x_2,\cdots \rangle}{\langle x_1x_2,x_1x_3, \cdots \rangle}$. The ring $R$ does not satisfy the \emph{ACC} on annihilators, as $\Ann(x_1,x_2,\cdots)\subseteq \Ann(x_2,x_3,\cdots)\subseteq \Ann(x_3,x_4,\cdots)\subseteq \cdots$. Also we claim that $R$ is an \emph{ACCP}-ring. To do this, first, notice that if $f,g\in R$ and $f_0,g_0\neq 0$, then
$\textit{deg}_{x_i}(fg)\geq \textit{max}(\textit{deg}_{x_i}(f),\textit{deg}_{x_i}(g)).$
Moreover, if $f_0=0$ and $\textit{deg}_{x_i}(g)>0$, then $\textit{deg}_{x_i}(fg)=0$ or $\textit{deg}_{x_i}(fg)>\textit{deg}_{x_i}(g)$ and
$(fg)_0=0$. The  case  $\textit{deg}_{x_i}(fg)>\textit{deg}_{x_i}(g)$ occurs if and only if $\textit{deg}_{x_i}(f), \textit{deg}_{x_i}(g)>0$.

Now let $\langle f_1 \rangle \subseteq \langle f_2 \rangle \subseteq \cdots$. So $f_n=g_nf_{n+1}$. If $g_{n,0}\neq 0$, $f_{n+1,0}\neq 0$ then $\textit{deg}_{x_i}(g)\geq \textit{deg}_{x_i}(f_{n+1})$ for each $i$. So $g_n$ must be a unit for each $n$ and since $k$ is a field, so the mentioned chain will stabilize.

Let $f_{1,0}=0$. Then $\textit{deg}_{x_i}(f_1g)\geq
\textit{deg}_{x_i}(f_1)$, when $\textit{deg}_{x_i}(f_1)>0$; and
$\textit{deg}_{x_i}(f_1g)=0$, otherwise. So one can see inductively
that $\textit{deg}_{x_i}(f_n)\leq \textit{deg}_{x_i}(f_1)$ if
$\textit{deg}_{x_i}(f_1)>0$ and $\textit{deg}_{x_i}(f_n)=0$,
otherwise. Hence
$\textit{deg}_{x_{i}}(f_n)=\textit{deg}_{x_{i}}(f_m)$ for each $x_i$
if $\textit{deg}_{x_i}(f_1)>0$ for infinitely many $m,n$ (otherwise
$\textit{deg}_{x_i}(f_1)=\infty$ which is impossible). Let $\langle
f_{n_1} \rangle \subseteq \langle f_{n_2} \rangle \subseteq \cdots$
be the mentioned chain such that
$\textit{deg}_{x_{i}}(f_{n_i})=\textit{deg}_{x_{i}}(f_{n_{i+1}})$.
So the chain stabilized by the fact that $k$ is a field. So the
chain $\langle f_1 \rangle \subseteq \langle f_2 \rangle \subseteq
\cdots$ will stabilize and $R$ is an \emph{ACCP}-ring.

Also, we claim that $R$ is reduced. Let $f\in R$ such that $f^2=0$. Let $\textit{deg}_{x_i}(f)>0$. It is easy to see that $\textit{deg}_{x_i}(f^2)=2\textit{deg}_{x_i}(f)>0$ which is impossible. So $R$ is reduced.

Now, we claim that $R[t]$ is not an \emph{ACCP}-ring. Let $f_n=(\sum_{i=n}^{\infty}x_i)t+(1+\sum_{i=n}^{\infty}x_i)$ and
$g_n(t)=x_{n-1}t+(1+x_{n-1})$. It is obvious that $f_ng_n=f_{n-1}$. So
$\langle f_1 \rangle \subseteq \langle f_2 \rangle \subseteq \cdots$.
Also it is easy to see that for each $n \geq 1$, $g_n \notin U(R)$.
Otherwise, let $g_n\in U(R)$ for some $n\ge 1$. So $g_nh=1$ for some $h \in R$ . So $(x_{n-1}t+(1+x_{n-1}))h=1$, and
$h_0=1+\sum_{j\neq n-1}a_jx_j$. So $h_0x_{n-1}+(1+x_{n-1})h_1=0$, and
\begin{align}\label{*}
x_{n-1}+h_1+x_{n-1}h_1=0.
\end{align}
Hence $h_{1,0}=0$. If $\textit{deg}_{x_{n-1}}(h_1)=0$, then left side of \ref{*} has a nonzero coefficient of $x_{n-1}$, but the right side does not. Else if, $\textit{deg}_{x_{n-1}}(h_1)>0$, then $\textit{deg}_{x_{n-1}}(x_{n-1}h_1)>\textit{max}(\textit{deg}_{x_{n-1}}(h_1),1)$. But $x_{n-1}h_1=-h_1-x_{n-1}$, which means that $\textit{deg}_{x_{n-1}}(-h_1-x_{n-1})=\textit{deg}_{x_{n-1}}(x_{n-1}h_1)$, which contradicts $\textit{deg}_{x_{n-1}}(x_{n-1}h_1)>\textit{max}(\textit{deg}_{x_{n-1}}(h_1),1)$. So $R[t]$ is not an \emph{ACCP}-ring. So the condition \emph{ACC} on annihilators can not be omitted.
\end{example}

\section{Archimedean Skew Generalized Power Series Rings}

A domain $R$ is said to be archimedean if $\cap_{n\geq 1}a^nR=0$
 for each nonunit $a$ of $R$. It is well-known that any domain satisfying \emph{ACCP} is archimedean, but the converse is not true (see, for example, Dumitrescu et al., \cite{Dumitrescu}, p. 1127).  We consider this property for skew generalized power series ring. First, recall that if $(S,.,\leq)$ is a strictly totally ordered monoid and $0\neq f \in R[[S,\omega]]$, then $\supp( f )$ is a nonempty well-ordered subset of $S$. The smallest element of $\supp( f )$ is denoted by $\pi( f )$.
R. Mazurek and M. Ziembowski in \cite[Proposition 3.2]{von} proved that if $s\in U(S)$ and $f (s)\in U(R)$, then $f \in U(R[[S,\omega]])$.

\begin{theorem}\label{domar}
Let $R$ be a right archimedean domain and $S$  an Artinian strictly totally ordered monoid. If $ \omega_s$ is injective for each $ s \in S$ and it preserves nonunit elements of $R$, then $ R[[S,\omega]] $ is a right archimedean domain.
\end{theorem}
\begin{proof}
Set $A=R[[S,\omega]]$. It is clear that  $A$ is a domain. Assume to
the contrary that $f$ is a nonunit element of $A$. So there is a
nonzero element $g$ in $\bigcap_{n\ge{1}}Af^n$. Then for each $n \in
\mathbb{N}$ there exists $h_n\in A$ such that $g=h_nf^n$. Using
\cite[Proposition 3.1(i)]{zim}, we get $\pi(g)=\pi(h_nf^n)$ and so
$\pi (g)=\pi(h_n)\pi(f^n)$. So
$g(\pi(g))=h_n(\pi(h_n))\omega_{\pi(h_n)}\left(f^n(\pi(f^n))\right)$.
There are two cases.

First, let $\supp(f)=\{1\}$, then $\pi(h_n)=\pi(g)$. So
$g(\pi(g))=h_n(\pi(h_n))\omega_{\pi(g)}\left(f^n(\pi(f^n))\right)$. Also $f^n(\pi(f^n))=\left(f(\pi(f))\right)^n$, since $R$ is a domain. So for each $n\in \mathbb{N}$,
\begin{align*}
g(\pi(g))=h_n(\pi(h_n))\left(\omega_{\pi(g)}\left(f(\pi(f))\right)\right)^n\in R\left(\omega_{\pi(g)}\left(f(\pi(f))\right)\right)^n.
\end{align*}
This yields that $g(\pi(g))\in \bigcap_{n\in\mathbb{N}}R\left(\omega_{\pi(g)}\left(f(\pi(f))\right)\right)^n
$. Also $f(\pi(f))$ is not a unit, since otherwise $\pi(f)$ and $f(\pi(f))$ would be both units and $f$ will be unit, by \cite[Proposition 3.1(i)]{zim}.

Since $R$ is right archimedean, as $g(\pi(g))=0$, which contradicts the fact that $\pi(g)\in \supp(g)$. Thus $g=0$ and the result follows.

Now suppose that $\supp(f)\neq \{1\}$. There are three cases.

Case 1. $\pi(f)>1$. We know that $\pi(g)=\pi(h_n)\pi(f)^n$.
So $\pi(h_n)<\pi(h_{n-1})$ for each $n$. Thus $\{\pi(h_n)\}$ forms a descending chain and must have a maximal element which is a contradiction.

Case 2. $\pi(f)<1$. Which means that $\{\pi(f^n)\}$ forms a descending chain and has a maximal element. This is also a contradiction.

Case 3. $\pi(f)=1$. So $f(\pi(f))$ is not a unit since otherwise $f$ would be a unit. Also one can see that
\begin{align*}
(h_1f)(1)=(h_2f)(1)=\cdots
\end{align*}
So $h_1f=0$, since $R$ is archimedean. So $h_if^i=0$ for all $i$ which contradicts our assumption.
\end{proof}
This result can be applied even when the endomorphisms $\omega_s$ are rigid monomorphisms. In fact, we proved the above theorem for $S$-rigid ring $R$ as follows.

\begin{theorem}\label{rigidar}
Let $S$ be an Artinian strictly totally ordered monoid and let $R$ be a right archimedean and $S$-rigid ring. Assume that $R$
satisfies the ACC on annihilators and that $\omega_s$ preserves nonunits of $R$ for each $s\in S$. Then $R[[S,\omega]]$ is
 a right archimedean reduced ring.
\end{theorem}
\begin{proof}
Assume to the contrary that $A=R[[S,\omega]]$ is not a right archimedean ring. Then there exists $g\in A$ such that $\bigcap_{n\in \mathbb{N}}Ag^n\neq \{0\}$.
Also consider $I_f=\{g(\pi(g))|g\in AfA\}$ for $f\in A$ as in Theorem \ref{domar}. There are two cases for $g$:

Case 1. $\Ann(\bigcup_{i\in \mathbb{N}}I_{a_i})=0$ for all $a_i\in A$ such that
\begin{align*}
0\neq a_1g=a_2g^2=\cdots =a_ng^n=\cdots.
\end{align*}
 So for each $i$, one can see that
\begin{align*}
\pi(a_ig^i)=\pi(a_i)\pi(g^i)=\pi(a_i)\pi(g)^i.
\end{align*}
Note that since $R$ is $S$-rigid, $A$ is reduced by (\cite{unified},\cite{frohnACC}) and $\pi(g^i)=\pi(g)^i$. Now, there are three cases:

$(i)$: $\pi(g)<1$. Then we have
\begin{align*}
1>\pi(g)>\pi(g^2)\cdots
\end{align*}
which should be stabilized and it is a contradiction.

$(ii)$: $\pi(g)>1$. Then, since $\pi(a_i)\pi(g^i)=\pi(a_{i+1})\pi(g^{i+1})$, we have
\begin{align*}
\pi(a_1)>\pi(a_2)>\cdots .
\end{align*}
Since $S$ is Artinian, $\pi(a_i)=\pi(a_{i+1})$ for some $i$. But this contradicts $\pi(g)>1$.

$(iii)$: $\pi(g)=1$, Then $g(\pi(g))$ is not a unit as $g$ is not a unit. So
\begin{align*}
a_i(\pi(a_i))g^i(\pi(g^i))=a_{i+1}(\pi(a_{i+1}))g^{i+1}(\pi(g^{i+1}))\quad for \quad i\in \mathbb{N}.
\end{align*}
So $a_1(\pi(a_1))g(\pi(g))\in Rg^i(\pi(g^i))$ for all $i\in \mathbb{N}$ and since $R$ is $S$-rigid and hence reduced, we have $a_1(\pi(a_1))g(\pi(g))\in R(g(\pi(g)))^i$. Also, $R$ is archimedean, and $a_1(\pi(a_1))g(\pi(g))\in \bigcap_{i\in \mathbb{N}}R(g(\pi(g)))^i$, so $a_1(\pi(a_1))g(\pi(g))=0$. So
\begin{align*}
0=a_1(\pi(a_1))g(\pi(g))=a_2(\pi(a_2))\left(g(\pi(g))\right)^2=\cdots.
\end{align*}
It is easy to show that $a_i(\pi(a_i))g(\pi(g))=0$ because of the fact that $R$ is reduced. So for all $i$, $g(\pi(g))\in \Ann\left(a_i(\pi(a_i))\right)$. So
\begin{align*}
g(\pi(g))\in \Ann\left(\bigcup_{i\in \mathbb{N}}I_{a_i}\right)=\{0\}.
\end{align*}
Hence $g(\pi(g))=0$ which is a contradiction.

Case 2. Let $\Ann\left(\bigcup_{i\in \mathbb{N}}I_{a_i}\right)\neq
\{0\}$  such that
\begin{align*}
0\neq a_1g=a_2g^2=\cdots =a_kg^k=\cdots.
\end{align*}
Define
\begin{align*}
T=\{\Ann\big(\bigcup_{i\in \mathbb{N}}I_{b_ig}\big)|\Ann\big(\bigcup_{i\in \mathbb{N}}I_{b_ig}\big)\neq 0, \quad 0\neq b_1g=\cdots =b_2g^2=\cdots \}.
\end{align*}
We know that $T\neq\emptyset$ and $R$ satisfies \emph{ACC} on annihilators. So $T$ has a maximal element like $\Ann\big(\bigcup_{i\in \mathbb{N}}I_{a_if}\big)$ for some $f\in A$ such that
\begin{align}\label{sharto}
0\neq a_1f=a_2f^2=\cdots =a_nf^n=\cdots .
\end{align}
We claim that $V:=\Ann\big(\bigcup_{i\in \mathbb{N}}I_{a_if}\big)$ is a completely prime ideal of $A$. Since $R$ is $S$-rigid, $V$ is a two sided ideal. Let $ab\in V$ and $a,b$ are not in $V$. Then
\begin{align*}
ab(a_if)(\pi(a_if))=0\quad for \quad all \quad i\in \mathbb{N}.
\end{align*}
So $a\in \Ann(I_{ba_if})$. But it is easy to see that
\begin{align}\label{shar}
V\subseteq \Ann(I_{ba_if}).
\end{align}
Since $b\notin V$, we can see that $\Ann(I_{ba_if})\neq \{0\}$. Also by multiplying b to equation (\ref{sharto}), we have
 \begin{align*}
0\neq ba_1f=ba_2f^2=\cdots \quad.
\end{align*}
So $\Ann(I_{ba_if})\in T$. Also $V$ is maximal in $T$. This and \ref{shar} yields that $V=\Ann(I_{ba_if})$, which implies that $a\in V$ which is a contradiction. So $V$ is a two sided completely prime ideal.

Now we show that $\omega_s$ is $V$-invariant for each $s\in S$. Let $r\in V$. Then $rI_{a_i}=0$ by Lemma \ref{kav}. So $V\subseteq \omega_s(V)$ for each $s\in S$. Now let $r\in \omega_s(V)$, so $r\in \omega_s(t)$ such that $t\in V$ and $tI_{a_i}=0$. So $\omega_s(tI_{a_i})=0$ which means that
$r\omega_s(I_{a_if})=\omega_s(t)\omega_s(I_{a_if})$.
So $r\in \Ann(\omega_s(I_{a_if}))$. But we have $V\subseteq\omega_s(V)$ which means that
\begin{align*}
\Ann(I_{a_if})\subseteq \omega_s\left(\Ann(I_{a_if})\right)=\Ann\left(\omega_s(I_{a_if})\right).
\end{align*}
 So $r\in \Ann(I_{a_if})=V$. Hence $\omega_s^{-1}(V)=V$ and $\omega_s$ is $V$-variant for each $s\in S$.

We know that $V$ is a completely prime ideal and hence the factor ring $U:=\frac{R}{V}$ is an archimedean domain and by Theorem \ref{domar}, $U[[S,\overline{\omega}]]$ is an archimedean domain.

We know that there exists $a_i\in A$ such that $a_1f=a_2f^2=\cdots$. Hence $\overline{a_1}\overline{f}=\overline{a_2}\overline{f^2}=\cdots$. We claim that $\overline{f}$ is nonunit in $W=U[[S,\overline{\omega}]]$. Otherwise, let $\overline{f}$ be a unit. So there exists $g\in A$ such that
$\overline{f}\overline{g}-\overline{1}=\overline{0}$; which means that
\begin{align*}
(fg-1)(\pi(fg-1))a_if(\pi(a_if))=0 \quad for \quad all \quad i\in\mathbb{N}.
\end{align*}
We have four cases.

$(i)$: $\pi(fg-1)=1_S$ and $(fg-1)(1)=-1_R$. Hence $\pi(f)\neq 1$ or $\pi(g)\neq 1$ which means $\overline{f}$ or $\overline{g}$ are not unit since $S$ is Artinian strictly totally ordered monoid, and the only unit of $S$ is $1$. So it is a contradiction.

$(ii)$: $\pi(fg)<1$. So
\begin{align*}
1>\pi(fg)>\pi(fg)^2>\cdots
\end{align*}
which contradicts to the fact that $S$ is Artinian.

$(iii)$: $\pi(fg-1)\neq 1$. So $fg=1+k$ for some $k\in W\setminus U$ and $\pi(k)>1$. So $f(\pi(f))\omega_{\pi(f)}\left(g(\pi(g))\right)=1$. Thus $f(\pi(f))$ is a unit and $\pi(f)=1$. This means that $f$ is a unit and it contradicts to the definition of $f$.

$(iv)$: Suppose that $fg(1)=r\neq 1$ and $\pi(fg)=1$ (i.e. $r\neq 0$). We know that $\overline{f}$ is a unit. Then $\overline{1}\in W\overline{f}^n$. So there exist $b_i\in A$, $l_i\in V$ such that
\begin{align}\label{shr}
1+l_0=b_1f+l_1=b_2f^2+l_2=\cdots .
\end{align}
By multiplying $(a_1f)(\pi(a_1f))$, it yields that
\begin{align*}
(a_1f)(\pi(a_1f))=(a_1f)(\pi(a_1f))b_1f=(a_1f)(\pi(a_1f))b_2f^2=\cdots .
\end{align*}
If $(a_1f)(\pi(a_1f))b_i(1)=t_i$ and $f(1)=r$, then $t_ir^{i}=0$ for some $i$. Hence $(t_ir)^{i}=0$ and since $R$ is reduced, $t_ir=0$.
So $(a_1f)(\pi(a_1f))b_i(1)f(1)=0$. So $\overline{b_i}(1)\overline{f}(1)=\overline{0}$ which means that $\overline{f}(1)=\overline{0}$ or $\overline{b_i}(1)=\overline{0}$. If $\overline{b_i}(1)=\overline{0}$, then $\overline{b_i}$ is not a unit which is impossible (We know that $\overline{b_i}\overline{f}^i=\overline{b_{i-1}}\overline{f}^{i-1}$ by equation \ref{shr}, so $\overline{b_i}\overline{f}^{i-1}=\overline{b_1}=\overline{1}$ which means that $\overline{b_i}$ is a unit). So $\overline{f}(1)=\overline{0}$ which is impossible according to our assumption,  $\overline{f}$ that is a unit.

So for each $i$, $t_ir^{i+1}\neq 0$ which means that $t_ir^{i+1}=t_{i+1}r^{i+2}$.
Since $R$ is archimedean, $\bigcap_{i}t_ir^{i+1}=0$ which contradicts to the fact that $\overline{b_i}$ is a unit for each $i$.

So $\overline{f}$ is a nonunit. We have $\overline{a_1}\overline{f}=\overline{a_2}\overline{f^2}=\cdots$.
Since $W$ is archimedean, $\overline{a_1}\overline{f}=\overline{0}$. If $\overline{a_1}=\overline{0}$, then $\overline{a_1}(s)=\overline{0}$ for every $s\in S$. So $a_1(s)(a_1f)(\pi(a_1f))=0$ for every $s\in S$. Also $\omega_t\left(a_1(s)(a_1f)\right)(\pi(a_1f))=0$ for $s,t\in S$. Hence, by Lemma \ref{kav},
\begin{align}
\left((a_1f)(\pi(a_1f))\right)&\omega_{\pi(a_1f)}\left((a_1f)(\pi(a_1f))\right)=(a_1f)^2(\left(\pi(a_1f)\right)^2)\nonumber\\
&=(a_1f)(\pi(a_1f))\omega_{\pi(a_1f)}\left(\sum_{xt=\pi(a_1f)}a_1(x)\omega_x(f(t))\right)\nonumber\\
&=\sum_{xt=\pi(a_1f)}\left((a_1f)(\pi(a_1f))\omega_{\pi(a_1f)}\left(a_1(x)\omega_x(f(t))\right)\right)=\sum_{xt=\pi(a_1f)}0=0.
\end{align}
So $\left((a_1f)(\pi(a_1f))\right)\omega_{\pi(a_1f)}\left((a_1f)(\pi(a_1f))\right)=0$ which yields $\left((a_1f)(\pi(a_1f))\right)^2=0$ by Lemma \ref{kav}. Thus $(a_1f)(\pi(a_1f))=0$ which is a contradiction. The case $\overline{f}=\overline{0}$ is impossible, similarly. This completes the proof.
\end{proof}

We provide an example of an   archimeadean \emph{ACCP}-ring $R$ with a homomorphism $\alpha$ which is not rigid, but $A=R[x;\alpha]$ is not an \emph{ACCPL}-ring. Note that,
an element $r=v_{i_1}v_{i_2}\cdots v_{i_k}$ is considered as a monomial with degree $k$. The degree of  $s=r_1+r_2+\cdots$ is $\max_{i\in \mathbb{N}}\textit{degree}(r_i)$.

\begin{example}
Suppose that $k$ is a field and consider the ring $R$ as follows:
\begin{align}
R=\{k+v_1k+v_2k+\cdots+v_1v_2k+v_1v_3k+\cdots |v_i^2=0,\quad v_iv_j=v_jv_i \quad \textit{for every i,j}\}.
\end{align}
In fact, $R$ also can be defined as follows
\begin{align}
R=\frac{k[x_1,x_2,\cdots ]}{(x_1^2,x_2^2,\cdots)k[x_1,x_2,\cdots]}.
\end{align}

Define $\alpha$ on $R$ given by

\begin{align}
\alpha(v_i)=\begin{cases}
v_{i+1}\quad ;   i \quad \textit{ a positive even number}; \\
0 \quad \quad ;    i \quad \textit{is odd}.
\end{cases}
\end{align}
It is obvious that $\alpha(1)=1$. We claim that $R$ satisfies \emph{ACCP}.

Now suppose that $R$ does not satisfy \emph{ACCP}. So there exists a non-stabilized chain:
\begin{align*}
a_1R\subseteq a_2R\subseteq \cdots , \quad \textit{with} \quad a_i\in R.
\end{align*}
One can see easily that $degree(a_i)>degree(a_{i+1})$. So $degree(a_1)=\infty$ which is a contradiction. So $R$ is an \emph{ACCP}-ring. Also we claim that $R$ is left archimeadean. Let there exists $a\in R$ such that
\begin{align*}
\bigcap_{n\in \mathbb{N}}Ra^n\neq \{0\}.
\end{align*}
Then there is $t\in Ra^n$ for each $n$ which means that the degree of $t$ should be $\infty$ and it is also a contradiction. So $R$ is  left archimeadean.

The homomorphism $\alpha$ is not rigid because $v_3\alpha(v_3)=0$, but $v_3\neq 0$.

We claim that $A=R[x;\alpha]$ is not an \emph{ACCPL}-ring. To do this, consider the following sequence:
\begin{align*}
f_0=a_0x+b_0 \quad, \quad f_n=a_nx+b_n,
\end{align*}
where
\begin{align*}
a_0=v_1+v_3+\cdots\quad , \quad b_0=1+v_2+v_4+\cdots
\end{align*}
and
\begin{align*}
a_n=(p_0)^{-1}\left(a_{n-1}-p_1\alpha((p_0)^{-1})\alpha(b_{n-1})\right)
\end{align*}
such that
\begin{align*}
p_0=1+v_1\quad and\quad p_1=v_1+v_3+\cdots.
\end{align*}
One can see that
\begin{align*}
f_{n-1}=(p_1x+p_0)f_{n}.
\end{align*}
So we get the following chain:
\begin{align}\label{zanjirmesal}
Rf_0\subseteq Rf_1\subseteq\cdots.
\end{align}
Note that $p_1$ is not nilpotent. Since degree of $p^k$ is $k$, so $p_1x+p_0$ is not a unit and we get
\begin{align*}
Rf_0\subsetneq Rf_1\subsetneq\cdots,
\end{align*}
which shows that $A$ does not satisfies \emph{ACCPL}.
\end{example}

\end{document}